\documentclass[12pt]{amsart}

\usepackage{amssymb,latexsym,amsmath,epsfig,amsthm,mathrsfs}
\usepackage{kpfonts}
\usepackage[T1]{fontenc}

\usepackage{rotating}
\usepackage{graphicx}
\usepackage{amssymb}
\usepackage{lineno}
\usepackage{enumitem}
\usepackage{cite}
\usepackage[top=3cm, bottom=3cm, left=3cm, right=3cm]{geometry}
\usepackage[usenames]{color}
\usepackage[colorlinks=true,
linkcolor=blue,
filecolor=blue,
citecolor=blue]{hyperref}

\newtheorem{theorem}{Theorem}[section]
\newtheorem{lemma}[theorem]{Lemma}

\newtheorem{prop}[theorem]{Proposition}

\theoremstyle{definition}

\theoremstyle{example}

\theoremstyle{remark}
\newtheorem{remark}[theorem]{Remark}

\numberwithin{equation}{section}
\numberwithin{theorem}{section}

\begin{document}
	
	\title[Extended inverse theorems for sumsets in integers]{Extended inverse theorems for $h$-fold sumsets in integers}
	
	
	\author[Mohan]{Mohan}
	\address{Department of Mathematics, Indian Institute of Technology Roorkee, Uttarakhand, 247667, India}
	\email{mohan98math@gmail.com}

	
	\author[R K Pandey]{Ram Krishna Pandey$^{\dagger}$}
	\address{Department of Mathematics, Indian Institute of Technology Roorkee, Uttarakhand, 247667, India}
	\email{ram.pandey@ma.iitr.ac.in}
	\thanks{$^{\dagger}$The corresponding author}

	
	
	

	\begin{abstract}
		Let $h \geq 2$, $k \geq 5$ be integers and $A$ be a nonempty finite set of $k$ integers. Very recently, Tang and Xing studied extended inverse theorems for  $hk-h+1 < \left|hA\right| \leq hk+2h-3$. In this paper, we extend the work of Tang and Xing and study all possible inverse theorems for  $hk-h+1<\left|hA\right| \leq hk+3h-4$. Furthermore, we give a range of $|hA|$ for which inverse problems are not possible.
	\end{abstract}
	
	\maketitle
	

	\section{Introduction}
	Let $A=\lbrace a_{0},a_{1}, \ldots, a_{k-1}\rbrace$ be a nonempty finite set of integers and $h$ be a positive integer. The \textit{$h$-fold sumset} is,  denoted by $hA$, defined  by
	\[hA:=\left\lbrace \sum_{i=0}^{k-1} \lambda_{i} a_{i}: \lambda_{i} \in \mathbb{N} \cup \left\lbrace 0\right\rbrace \ \text{for} \ i= 0, 1, \ldots, k-1 \ \text{with} \  \sum_{i=0}^{k-1} \lambda_{i}=h\right\rbrace.
	\]
	Throughout the paper, $\mathbb{N}$ denotes the set of all positive integers and  $|A|$ denotes the cardinality of the set $A$. For integers $\alpha$ and $\beta$, let
	\begin{align*}
		\alpha * A &= \lbrace \alpha a : a\in A \rbrace,\\
		A + \beta &= \lbrace a + \beta  : a \in A \rbrace.	
	\end{align*}
	For $\alpha \leq \beta$, we let $[\alpha, \beta] = \lbrace \alpha, \alpha +1, \ldots, \beta \rbrace$.
	The greatest common divisor of the integers $x_{1}, x_{2}, \ldots, x_{k}$ is denoted by $\left(x_{1},x_{2}, \ldots, x_{k}\right)$. Let $A = \{a_0, a_1, \ldots, a_{k-1}\}$ be a set of integers  with $a_0 < a_1 < \cdots < a_{k-1}$. Then, we define
	\[ d(A) := (a_{1}-a_{0}, a_{2}-a_{0}, \ldots, a_{k-1}-a_{0})\]  and  \[A^{(N)} :=  \left\lbrace a^{\prime}_{i}= \dfrac{a_{i}-a_{0}}{d}: a_{i} \in A \right\rbrace.\]
	The set $A^{(N)}$  is called the normal form of $A$. We have $d(A^{(N)})=1$. The  $h$-fold sumset $hA$ is translation and dilation invariant of set $A$, that is,
	\[h\left(\left(\alpha *A\right) + \beta \right) =(\alpha * \left(hA \right) ) + h\beta.
	\]
	
	Various types of sumsets in groups are being studied from more than two centuries. The study of  minimum cardinality of sumset is called the \textit{direct problem}  and  the characterization of the underlying set for the minimum cardinality of sumset is known as the \textit{inverse problem} in the area of additive number theory or additive combinatorics. Further, characterization of the underlying set for small deviation from the minimum size of the sumset is called the \textit{extended inverse theorem}.  The following are classical results for minimum cardinality of $h$-fold sumset $hA$ and for the associated inverse problems.
	
	\begin{theorem}\textup{\cite[Theorem 1.4, Theorem 1.6]{MBN1996}}\label{THeorem Nathanson}
		Let $h\geq 1$ and $A$ be a nonempty finite set of integers. Then
		\begin{center}
			$\left| hA\right| \geq h\left| A\right|  - h + 1$.
		\end{center}
		This lower bound is best possible. Furthermore, if  $|hA|$ attains this lower bound with $h \geq 2 $, then $A$ is an arithmetic progression.
	\end{theorem}
	
	Freiman \cite{GAF1959, GAF1973} proved some direct and extended inverse theorem for the $2$-fold sumset $2A$ which are mentioned below in  Theorem \ref{Theorem Freiman I}  and Theorem \ref{Theorem Freiman II}, respectively.
	
	\begin{theorem}\textup{\cite[Theorem 1.10]{GAF1973}}\label{Theorem Freiman I}
		Let $k \geq 3$ and  $A=\lbrace a_{0}, a_{1}, \ldots, a_{k-1}\rbrace$ be a set of integers such that
		$0=a_{0}<a_{1}< \cdots < a_{k-1}$ with $d(A)=1$.
		\begin{itemize}
			\item[\upshape(a)] If $a_{k-1} \leq  2k-3$, then $\left|2A\right| \geq k +  a_{k-1}$.
			\item[\upshape(b)] If \  $a_{k-1} \geq  2k-2$, then $\left|2A\right| \geq 3k-3$.
		\end{itemize}
	\end{theorem}

	\begin{theorem}\textup{\cite[Theorem 1.9]{GAF1973}}\label{Theorem Freiman II}
		Let $A$ be a finite set of $k \geq 3$ integers.
		If  $\left|2A\right|=2k-1+b<3k-3$, then $A$ is a subset of an arithmetic progression of length atmost  $k+b$.
	\end{theorem}
	
	Lev's \cite{VFL1996} generalization of Theorem \ref{Theorem Freiman I}  for the $h$-fold sumset $hA$ is as follows.
	
	\begin{theorem}\textup{\cite[Theorem 1]{VFL1996}}\label{Theorem Lev}
		Let $h,k \geq 2$ be integers. Let $A= \lbrace a_{0}, a_{1}, \ldots, a_{k-1} \rbrace$ be a set of integers such that $0=a_{0}<a_{1}< \cdots< a_{k-1}$ and $d(A)=1$. Then  $$ \left|hA \right| \geq |(h-1)A|+ \min\lbrace a_{k-1},h(k-2)+1 \rbrace.$$
	\end{theorem}
	
	Recently, Tang and  Xing \cite{TANG2021} proved some extended inverse theorems for $hA$, where $h \geq 2$. Some other generalizations of Theorem \ref{Theorem Freiman I}, can also be seen in \cite{TANG2021}. The two main extended inverse theorems of Tang and  Xing \cite{TANG2021} are mentioned below.
	
	\begin{theorem}\textup{\cite[Theorem 1.1]{TANG2021}} \label{Theorem Tang I}
		Let $h \geq 2$ and $k \geq 5$ be integers. Let $A$ be a set of integers with $\left|A\right|=k$. If $hk-h+1< \left|hA\right| \leq hk+h-2$, then $$ A^{(N)}=[0,k] \setminus \{x\} \text{ for } 1\leq x \leq k-1.$$
		Moreover, $\left|hA\right|=hk$ for $x=1$ or $k-1$, and $\left|hA\right|=hk+1$ for $2\leq x \leq k-2$.
	\end{theorem}
	
	\begin{theorem}\textup{\cite[Theorem 1.2]{TANG2021}} \label{Theorem Tang II}
		Let $h \geq 2$ and $k \geq 5$ be integers. Let A be a set of integers  with $\left|A\right|=k$. If $hk+h-2<\left|hA\right| \leq hk+2h-3$, then \[A^{(N)}=[0,k+1] \setminus \{x,y\}, 1\leq x < y \leq k+1.\]
		
		Moreover, we have
		\begin{itemize}
			\item[\upshape(a)] $\left|hA\right|= hk+h-1$ for $\lbrace x,y \rbrace=\lbrace 1,2 \rbrace, \lbrace k-1,k \rbrace, \lbrace 1,k \rbrace, \lbrace 1,3  \rbrace, \lbrace k-2,k \rbrace;$
			
			\item[\upshape(b)] $\left|hA\right|=hk+h$ for $x=1$ and $4\leq y \leq k-1$ when $h\geq 2$; or $2\leq x \leq k-3$ and $y=k$ when $h\geq 2$, or $\{x,y\} = \{2,3\}, \{ k-2,k-1 \}$ when $h=2$;
			
			\item[\upshape(c)] $\left|hA\right|=hk+h+1$ for $2\leq x < y \leq k-1$, except for $\{x,y\} = \{2,3\}, \{ k-2,k-1 \}$ when $h=2$.
		\end{itemize}
	\end{theorem}
	
	Conclusions from Theorem \ref{Theorem Tang I} and Theorem \ref{Theorem Tang II} are the following:
	\begin{enumerate}
	\item $\left|hA\right|=hk$ if and only if $A^{(N)}= [0,k] \setminus \{x\}$, where $x\in \{1,k-1\}$,
	
	\item for $h\geq 3$,  $\left|hA\right|=hk+1$ if and only if  $A^{(N)}=[0,k] \setminus \lbrace x \rbrace$, where $2\leq x \leq k-2$,
	
	\item $\left|2A\right|=2k+1$ if and only if $A^{(N)}=[0,k+1] \setminus \lbrace x,y \rbrace$, where $\{x,y\}$ is one of the sets $\{1,2\}$, $\{k-1,k\}$, $\{1,k\}$, $\{1,3\}$, $\{k-2,k\}$, and $\{i,k+1\}$ where $2\leq i \leq k-2$,
	
	\item for $h\geq 3$, $\left|hA\right|=hk+h-1$ if and only if $A^{(N)}=[0,k+1] \setminus \{x,y\}  $, where  $ \{x,y\}$ is one of the sets  $\{1,2\}, \{k-1,k\}, \{1,k\}, \{1,3\}$, and  $\{k-2,k\}$,
	
	\item for $h\geq 3$, $\left|hA\right|=hk+h$ if and only if $A^{(N)}=[0,k+1] \setminus \{ x,y \}$, where $\{x,y\}$ is one of the sets $\{i,k\}$ where $2 \leq i \leq k-3$, and  $\{1,j\}$ where $4 \leq j \leq k-1$,
	
	\item for $h\geq 4$, $\left|hA\right|=hk+h+1$ if and only if $A^{(N)}=[0,k+1] \setminus \{ x, y \}$, where $2\leq x < y \leq k-1$.
\end{enumerate}

	In this paper, we study possible inverse problems, for $hk+h \leq \left|hA\right| \leq hk+2h+1$. Our main result, Theorem \ref{Main Theorem},  extends the work of Tang and Xing \cite{TANG2021}.

	\begin{theorem}\label{Main Theorem}
			Let $h \geq 2$ and $k \geq 5$ be positive integers. Let $A=$ be a set of $k$ integers with $\min(A) = 0$ and  $d(A)=1.$ Then
		\begin{enumerate}
			\item[\upshape(1)] for $k\geq 6$, $\left|2A\right|= 2k+2$ if and only if $A=[0,k+2]\setminus\{x,y,z\}$, where  $\{x,y,z\}$ is one of the sets $\{2,3,k+2\}$, $\{k-2,k-1,k+2\}$, $\{1,i,k+2\}$ where $4\leq i \leq k-1$, $\{i,k,k+2\}$ where $2\leq i \leq k-3$, $\{1,2,3\}$, $\{k-1,k,k+1\}$, $\{1,2,4\}$, $\{k-2,k,k+1\}$, $\{1,2,k+1\}$, $\{1,k,k+1\}$,  $\{1,2,5\}$, $\{k-3,k,k+1\}$, $\{1,3,5\}$, $\{k-3,k-1,k+1\}$, $\{1,3,k+1\}$, $\{1,k-1,k+1\}$, and $\{2,3,k+2\}$.
			
			\item[\upshape(2)] $\left|3A\right|=3k+4$ if and only if  $A=[0,k+2] \setminus\{x,y,z\}$, where  $\{x,y,z\}$ is one of the sets $\{1,2,3\}$, $\{k-1,k,k+1\}$,   $\{1,2,4\}$, $\{k-2,k,k+1\}$, $\{1,2,k+1\}$, $\{1,k,k+1\}$,  $\{1,2,5\}$, $\{k-3,k,k+1\}$, $\{1,3,5\}$, $\{k-3,k-1,k+1\}$, $\{1,3,k+1\}$, $\{1,k-1,k+1\}$, and 
			$\{i,j,k+2\}$ where $2 \leq i < j \leq k-1$.
			
			\item[\upshape(3)] for $h \geq 4$, $\left|hA\right| = hk+2h-2$ if and only if $A=[0,k+2] \setminus \{x,y,z\}$, where  $\{x,y,z\}$ is one of the sets $\{1,2,3\}$, $\{k-1,k,k+1\}$,  $\{1,2,4\}$, $\{k-2,k,k+1\}$, $\{1,2,k+1\}$, $\{1,k,k+1\}$, $\{1,2,5\}$, $\{k-3,k,k+1\}$,  $\{1,3,5\}$, $\{k-3,k-1,k+1\}$, $\{ 1,3,k+1 \}$, and $\{1,k-1,k+1\}$.
			
			\item[\upshape(4)] for  $h \geq 4$ and $k \geq 5$ or $h \geq 3 $ and $k \geq 6$ , $\left|hA\right|= hk+2h-1$ if and only if $A=[0,k+2] \setminus \{x,y,z\}$, where  $\{x,y,z\}$ is one of the sets $\{1,3,4\}$, $\{k-2,k-1,k+1\}$,  $\{1,2,i\}$ where $6\leq i \leq k$, $\{i,k,k+1\}$ where $2 \leq i \leq k-4$,  $\{1,3,i\}$ where $6\leq i \leq k$, $\{i,k-1,k+1\}$ where $2 \leq i \leq k-4$, and  $\{1,i,k+1\}$ where $4\leq i \leq k-2$.
			
			\item[\upshape(5)] for $h \geq 4$ and $k \geq 6$ or $h \geq 5$ and $k \geq 5$,  $\left|hA\right|= hk+2h$ if and only if $A=[0,k+2] \setminus \{x,y,z\}$, where  $\{x,y,z\}$ is one of the sets $\{2,3,k+1\}$, $\{1,k-1,k\}$, $\{i,i+1,k+1\}$ where $ 3 \leq i \leq k-4$, $\{1,i,i+1\}$ where $5 \leq i \leq k-2$,   $\{i,i+2,k+1\}$ where $2 \leq i \leq k-4$, $\{1,i,i+2\}$ where $4 \leq i \leq k-2$,  $\{1,i,j\}$ where $4\leq i \leq j-3 \leq k-3$, and $\{i,j,k+1\}$ where $2 \leq i \leq j-3 \leq k-5$.
			
			\item[\upshape(6)] for $h \geq 5$ and $ k \geq 6$ or $h \geq 6$ and $k \geq  5$, $ \left|hA\right| = hk+2h+1$ if and only if $A=[0,k+2] \setminus \{x,y,z\}$, where $2\leq x < y < z \leq k$
			
		\end{enumerate}
	\end{theorem}	
	
	To prove Theorem \ref{Main Theorem}, we prove first some lemmas (Section \ref{lemma}) and Propositions (Section \ref{prop}). In Section \ref{Mainsection} , we prove Theorem \ref{Main Theorem} and give some concluding remarks about the cases where the extended inverse theorems are not possible in Section \ref{conclusion}.

	\section{Lemmas}\label{lemma}

	\begin{lemma}\label{Lemma I}
			Let $h\geq 2$ and $k\geq 5$ be positive integers. Let $A$ be a set in normal form having $k$ integers.
		\begin{enumerate}
			\item[\upshape(a)] If  $k \geq 6$   and $\left|hA\right|=hk+3h-4$, then $A \subseteq [0,k+2]$.
			
			\item[\upshape(b)] If   $\left|hA\right|< hk+3h-4$,  then $A \subseteq [0,k+2]$.
		\end{enumerate}
	\end{lemma}

\begin{proof}
By Theorem \ref{Theorem Lev}, we have
\begin{align}\label{Lev-eq-1}
	|hA|& \geq |(h-1)A|+ \min\{a_{k-1},h(k-2)+1\} \nonumber\\
	& \geq  |(h-2)A|+ \min\{a_{k-1},h(k-2)+1 \} +\min \{a_{k-1},(h-1)(k-2)+1\} \nonumber\\
	& \vdots \nonumber \\
	& \geq |A|+ \min\{a_{k-1}, h(k-2)+1\}+ \cdots+ \min \{a_{k-1},2(k-2)+1\}.
\end{align}
Assume $k \geq 6$ and $\left|hA\right| = hk+3h-4$. If $a_{k-1} \geq 2k-2$,  then  from \ref{Lev-eq-1}, we have $$\left|hA\right| \geq k+(h-2)(2k-2)+2k-3>hk+3h-4,$$
which is not possible. Therefore $a_{k-1} \leq 2k-3$.  Again  using \ref{Lev-eq-1} and $a_{k-1} \leq 2k-3$, we obtain $$hk+3h-4 = \left|hA\right| \geq k+(h-1)a_{k-1},$$
which gives $a_{k-1}\leq k+2$. This proves $(a)$. 

Assume  $\left|hA\right|<hk+3h-4$.  If $a_{k-1} \geq 2k-2$,  then  from \ref{Lev-eq-1}, we have $$\left|hA\right| \geq k+(h-2)(2k-2)+2k-3\geq hk+3h-4,$$ which is not possible. Therefore $a_{k-1} \leq 2k-3$. Again using  \ref{Lev-eq-1} and $a_{k-1} \leq 2k-3$, we obtain  $$hk+3h-4 > \left|hA\right| \geq k+(h-1)a_{k-1},$$
which gives $a_{k-1}\leq k+2$. This proves $(b)$.
	\end{proof}
	
	\begin{lemma}\label{Lemma II}
		Let $A=[0,x-1] \cup [x+r,y]$, where $x \geq 2$, $r\geq 0$, $x+r \leq y-1$, and $\left|A\right| \geq 4$. If $h\geq r+1$, then $hA=[0,hy]$.
	\end{lemma}
	
	\begin{proof}
		Clearly, $hA \supseteq [0, h(x-1) ] \cup [h(x+r), hy]$. Let $A_{1}=\lbrace x-2, x-1 \rbrace$ and $A_{2}=\lbrace x+r,x+r+1 \rbrace$ be subsets of $A$. Then $h(A_{1} \cup A_{2}) \subseteq hA$.	We have
		\begin{align*}
			h(A_{1} \cup A_{2})&= \bigcup_{l=0}^{h} \left( \left( h-l\right) A_{1} + lA_{2}\right) \\
			&= \bigcup_{l=0}^{h} \left( [(h-l)(x-2),(h-l)(x-1)]+[l(x+r),l(x+r+1)] \right)\\
			& = \bigcup_{l=0}^{h} [hx-2h+l(r+2), hx-h+l(r+2)].
		\end{align*}
		If $h \geq r+1$, then $hx-2h+(l+1)(r+2) \leq hx-h+l(r+2)+1$. So
		$$h(A_{1} \cup A_{2})=\bigcup_{l=0}^{h} [hx-2h+l(r+2), hx-h+l(r+2)]=[hx-2h,hx+h(r+1)].$$
		Hence, $hA=[0,hy]$.
	\end{proof}

	Now we generalize Lemma \ref{Lemma II} in Lemma \ref{Lemma III}.
	
	\begin{lemma}\label{Lemma III}
		Let $x$, $t$, $r$, and $y$ be integers such that $2 \leq t \leq x$, $r \geq 0$, and $x+r \leq y-t+1$.	If $A=[0,x-1] \cup [x+r,y]$, and $h \geq  \dfrac{r+t-1}{t-1}$, then $hA=[0,hy]$.
	\end{lemma}
	
	\begin{proof}
		Clearly, $[0, h(x-1) ] \cup [h(x+r), hy] \subseteq hA $. Let $A_{1}=\{ x-t, \ldots, x-2, x-1 \}$ and $A_{2}=\{ x+r, x+r+1, \ldots, x+r+t-1 \}$ be subsets of $A$. Then $h(A_{1} \cup A_{2}) \subseteq hA$, where
		\begin{align*}
			h(A_{1} \cup A_{2})&= \bigcup_{l=0}^{h} \left( \left( h-l\right) A_{1} + lA_{2}\right) \\
			&= \bigcup_{l=0}^{h} \left( [(h-l)(x-t),(h-l)(x-1)]+[l(x+r),l(x+r+t-1)] \right)\\
			& = \bigcup_{l=0}^{h} [hx-ht+l(r+t), hx-h+l(r+t)].
		\end{align*}
		Since $h \geq  \dfrac{r+t-1}{t-1}$, we have $hx-ht+(l+1)(r+t) \leq hx-h+l(r+t)+1$. Therefore,\[h(A_{1} \cup A_{2})=\bigcup_{l=0}^{h} [hx-ht+l(r+t), hx-h+l(r+t)]=[hx-ht,hx+h(r+t-1)].\] Hence, $hA=[0,hy]$.
	\end{proof}
	
	\begin{remark}
		Putting $t=2$ in Lemma \ref{Lemma III} we get Lemma \ref{Lemma II}.
	\end{remark}	
	For the upcoming lemmas and propositions, we assume $r(m,n)$ be the least nonnegative residue of $m$ modulo $n$, where $m$ and $n$ are positive integers.
	\begin{lemma}\label{Lemma IV}
		Let $A= \{ a_{0},a_{1}, \ldots, a_{k-1} \}$ be a set of integers such that $0=a_{0}<a_{1}<\cdots < a_{k-1}$. If $0<a_{t} \leq m \leq ha_{t}-1$, then $m\in hA$, provided $a_{t}+r(m,a_{t})\in 2A$.
	\end{lemma}
	\begin{proof}
		Since $0<a_{t} \leq m \leq ha_{t}-1$, we have  $1 \leq \lfloor \frac{m}{a_{t}}\rfloor \leq h-1$. Also, $a_{t}+r(m,a_{t}) = a_{i}+a_{j} \in 2A$,  for some $i,j$. So, we can write $m$ as
		\[
		m =
		\underbrace{a_{t}+ \cdots+ a_{t}}_{{\lfloor \frac{m}{a_{t}}\rfloor-1} ~ times} + a_{i}+a_{j}  +\underbrace{0+ \cdots +0}_{h-1-\lfloor \frac{m}{a_{t}}\rfloor ~ times}.
		\]
		Hence, we get the lemma.
	\end{proof}

	\section{Propositions}\label{prop}
	
	\begin{prop}\label{Proposition I}
		Let $x$, $y$ and $h$ be positive integers with $y \geq 2x-1$.
		\begin{enumerate}
			\item[\upshape(a)] If $A = \{0\}  \cup [x,y]$,  then $hA= \{0\}  \cup [x,hy]$.
			
			\item[\upshape(b)] If $A =  [0,y-x] \cup \{y\}$, then $hA=  [0,hy-x] \cup \{hy\}$.
		\end{enumerate}
	\end{prop}
	\begin{proof}
		Let $A = \{0\} \cup [x,y]$. Then $\{0\} \cup [x,y] \cup [hx,hy] \subseteq hA$. Next, we show $[x,hx-1] \subseteq hA$. Since $y\geq 2x-1$, we have $$x+r(m,x) \in 2A,$$ where $m \in [x,hx-1]$.  
		Therefore, by Lemma \ref{Lemma IV}, we get $[x,hx-1] \subseteq hA$. Thus, $hA = \{0\} \cup [x,hy]$.  This proves $(a)$. Now,  if $A = [0,y-x] \cup \{y\}$, then $A = y - (\{0\} - [x,y])$. Therefore,  $$hA = h([0,y-x] \cup \{x\}) = h(y-(\{0\} \cup [x,y])) = hy - (\{0\} \cup [x,hy]) = [0,hy-x] \cup \{hy\}.$$ This completes the proof of the proposition.
	\end{proof}
	Now, we compute $|hA|$ when $A=[0,k+2] \setminus \{ x, y, z \}$ with $|A|=k \geq 5$. Since $|hA|$ is translation  invariant, so we assume $\min(A) = 0$. 
	If $z=k+2$, then $A=[0,k+1] \setminus \{ x, y \}$ and this was already studied by Tang and Xing \cite{TANG2021}. Therefore,  we  assume $a_{0} = 0$, $a_{k-1}=k+2$, and $1\leq x <y< z\leq k+1$ in $A = \{a_{0}, a_{1}, \ldots,a_{k-1}\} = [0,k+2] \setminus \{x,y,z\}$.

	\begin{prop}\label{Proposition II}
		Let $ h\geq 2$ and $k \geq 5$ be positive integers. Let \[A =[0,k+2] \setminus \{ x,x+1,x+2 \},\]  where $x \in [1,k-1]$. 
		\begin{enumerate}
			\item[\upshape(1)] If $x\in \{1,k-1\}$, then $\left|hA\right|=h(k+2)-2$.
			
			\item[\upshape(2)] If $x\in\{2,k-2\}$, then
			\[
			\left|hA\right| =
			\begin{cases}
				h(k+2)+1, & \text{ if } h \geq 4~\text{and}~ k\geq 5; \\
				3k+6,   & \text{ if }  h=3~\text{and}~  k\geq 5; \\
				2k+3,     & \text{ if }  h=2~\text{and}~ k\geq 6;\\
				12 ,      & \text{ if } h=2~\text{and}~  k= 5.
			\end{cases}
			\]
			
			\item[\upshape(3)] If $x \in [3,k-3]$, then
			\[
			\left|hA\right| =
			\begin{cases}
				h(k+2)+1, & \text{if $h\geq 3$ and  $k \geq 6$}; \\
				2k+5,     &  \text{if $h=2$, $ x \in [4,k-4]$ and  $k\geq 8$}; \\
				2k+4,   & \text{if $h=2$, $x\in \{3, k-3\}$ and  $k\geq 6$}.
			\end{cases}
			\]
		\end{enumerate}
	\end{prop}
	
	\begin{proof}
		\begin{enumerate}
			\item [\upshape(1).] 
			\begin{enumerate}
				\item [\upshape(a)] If  $x=1$, then $A=\{0\} \cup [4,k+2]$. By Proposition \ref{Proposition I}, we have
				\[hA=\{0\} \cup [4, h(k+2)],\]
				which gives
				\[|hA|=h(k+2)-2.\]  
				
				\item [\upshape(b)]  If $x=k-1$, then we have $A = [0,k-2] \cup \{k+2\} = \{k+2\} - ( \{0\} \cup [4,k+2])$. Since cardinality of $h$-fold sumset is translation invariant, we have \[\left|hA\right|= \left|h(\{k+2\} - ( \{0\} \cup [4,k+2]))\right| = \left|h( \{0\} \cup [4,k+2])\right|= h(k+2)-2.\]
			\end{enumerate}
			\item [\upshape(2).] \begin{enumerate}
				\item [\upshape(a)]  If   $x=2$, then $A=[0,1] \cup [5,k+2]$. For $h=2$, it is easy to see that 
				\[
				2A =
				\begin{cases}
					[0,2] \cup [5,2(k+2)], & \text{if $k \geq 6$}; \\
					[0,2] \cup [5,8] \cup [10,14], & \text{if $k=5$}. \\
				\end{cases}
				\] 
				Now, assume $h \geq 3$.  Clearly, $ [0,h] \cup [5h,h(k+2)] \subseteq hA$. Next, we show that $[5,5h-1] \subseteq hA$. Suppose $m$ be a positive integer such that $5 \leq m \leq 5h-1$. Then \[1 \leq \Big\lfloor \frac{m}{5}\Big\rfloor \leq h-1 \  \text{and} \  r(m,5) \in [0,4].\]  Note that,  $5+r(m,5) \in 2A$ when  $r(m,5) \in [0,3]$.  Lemma \ref{Lemma IV} implies that $m \in hA$ for $r(m,5) \in [0,3]$. For  $r(m,5)=4$ and   $\lfloor \frac{m}{5}\rfloor \geq 2$, we have  
				\[ m= \underbrace{5+ \cdots +5}_{\lfloor \frac{m}{5}\rfloor -2~times} +7+7+ \underbrace{ 0+ \cdots +0}_{h-{\lfloor \frac{m}{5}\rfloor}~times} \in hA.\]
				For $r(m,5)=4$ and  $\lfloor \frac{m}{5}\rfloor=1$, we get  $m = 9$, where
				$9=7+1+1 \in hA$.  Therefore $hA=[0,h] \cup [5,h(k+2)]$ for all $h \geq 3$.
				Hence,
				\[
				hA =
				\begin{cases}
					[0,h(k+2)],                    & \text{ if }  h \geq 4 ~\text{and}~ k\geq 5;  \\
					[0,3] \cup [5,3(k+2)],          & \text{ if }  h=3~ \text{and}~  k\geq 5; \\
					[0,2] \cup [5,2(k+2)],          & \text{ if }  h=2 ~\text{and}~ k\geq 6;\\
					[0,2] \cup [5,8] \cup [10,14],  & \text{ if }  h=2~ \text{and}~  k= 5,
				\end{cases}
				\]
				and
				\[
				\left|hA\right| =
				\begin{cases}
					h(k+2)+1,  & \text{ if }  h \geq 4 \text{ and } k\geq 5; \\
					3(k+2),    & \text{ if }  h=3\text{ and }  k\geq 5; \\
					2k+3,      & \text{ if }  h=2\text{ and } k\geq 6;\\
					12,        & \text{ if }  h=2\text{ and }  k= 5.
				\end{cases}
				\]
				\item [\upshape(b)] If $x=k-2$, then we can write $A=(k+2)- ([0,1] \cup [5,k+2])$, which is a translation of $[0,1] \cup [5,k+2]$, the set for which we have the result. So, we have the result for $x=k-2$ also, as  $\left|hA\right|$ is translation invariant.
			\end{enumerate}
			\item  [\upshape(3).] If $3\leq x \leq k-3$, then $A=[0,x-1] \cup [x+3,k+2]$ and $k \geq 6$.	Using  Lemma \ref{Lemma III} (for $t=r=3$), we obtain $$hA=[0,h(k+2)] ~ \text{for}~h\geq 3.$$\\
			\hspace{0.5 cm} Now, assume $h=2$. If $x=3$,  then $A=[0,2] \cup [6,k+2]$, and if $x=k-3$, then $A= (k+2)-([0,2] \cup [6,k+2])$. For $A=[0,2] \cup [6,k+2]$, we have $2A=[0,4] \cup [6,2(k+2)]$ and hence $\left|2A\right|=2(k+2)$. Similarly, if $A= (k+2)-([0,2] \cup [6,k+2])$, then $\left|2A\right|=2(k+2)$.
			If $4\leq x\leq k-4$, then $A=[0,x-1] \cup [x+3,k+2]$ and $k \geq 8$.  Using Lemma \ref{Lemma III} (for $r=3, t=4$),  we have  $2A=[0,2(k+2)]$. Hence
			\[|hA| =
			\begin{cases}
				h(k+2)+1, & \text{if $h\geq 3$ and  $k \geq 6$}; \\
				2k+5, &  \text{if $h=2$, $ x \in [4,k-4]$ and  $k\geq 8$}; \\
				2k+4,  & \text{if $h=2$, $x\in\{3, k-3\}$ and  $k\geq 6$}.
			\end{cases}\]
			
		\end{enumerate}
		This completes the proof of the proposition.
	\end{proof}
	
	Now, if exactly two of $x$, $y$ and $z$ are consecutive integers, then  it is sufficient to assume $x$ and $y$ are consecutive integers, and for the case when $y$ and  $z$ are consecutive integers, we take $A=(k+2) - ([0,k+2] \setminus \{x,x+1,z\})$.
	
	\begin{prop}\label{Proposition III}
		Let $h\geq 2$ and  $k \geq 5$  be positive integers. Let \[A=[0,k+2] \setminus \{x,y,z\},\] where $y=x+1$ and $1 \leq x \leq z-3 \leq k-2$.
		
		\begin{enumerate}
			\item[\upshape(1)] If   $\{ x,y,z \}=\{ 1,2,4 \}$,  then $\left|hA\right|=h(k+2)-2$.
			
			\item[\upshape(2)]  If  $\{ x,y,z \}=\{ 1,2, k+1 \}$,   then $\left|hA\right|=h(k+2)-2$.
			
			\item[\upshape(3)] If  $\{ x,y,z \}=\lbrace 1,2,i \rbrace$, where $5 \leq i \leq k$,   then
			\[
			|hA| =
			\begin{cases}
				h(k+2)-2, & \text{if $i=5 $};\\
				h(k+2)-1, & \text{if $i \not= 5$}.	
			\end{cases}
			\]
			
			\item[\upshape(4)] If  $\{ x,y,z\}=\{2,3,5\}$, then  	
			\[
			\left|hA\right| =
			\begin{cases}
				h(k+2)+1, & \text{if $h \geq 3$ and $k \geq 5$};\\
				2k+4, & \text{if $h=2$ and $k\geq 6$};\\
				13, & \text{if $h=2$ and $ k=5$}.
			\end{cases}
			\]
			
			\item[\upshape(5)] If  $\{ x,y,z\}=\{ 2,3,k+1 \}$, then 	\[
			\left|hA\right| =
			\begin{cases}
				
				h(k+2), & \text{if $h \geq 3$ and $k \geq 5$};\\
				2k+3, & \text{if $h=2$ and $ k\geq 5$}.
			\end{cases}
			\]
			
			\item[\upshape(6)] If  $\{ x,y,z \}=\{2,3,i\}$, where $6\leq i \leq k$, then \[
			\left|hA\right| =
			\begin{cases}	
				h(k+2)+1, & \text{if $h \geq 3$ and $k \geq 6$};\\
				2k+4, & \text{if $h=2$ and $ k\geq 6$}.
			\end{cases}
			\]
			
			\item[\upshape(7)] If  $\{x,y,z\}=\{k-2,k-1,k+1\}$, then $\left|hA\right|=h(k+2)-1$.
			
			\item[\upshape(8)] If   $\{x,y,z\}=\{i,i+1, k+1\}$, where  $3 \leq i \leq k-3$,  then $\left|hA\right|=h(k+2)$.
			
			\item[\upshape(9)] If $\{x,y,z\}=\{i,i+1,i+3\}$, where $ 3 \leq i \leq k-3$, then  	\[ |hA| = \begin{cases}
				h(k+2)+1, & ~\text{if}~ h\geq 3;\\
				2k+5, & ~\text{if}~ h=2 ~\text{and}~ 3\leq i \leq k-4;\\
				2k+4, &~\text{if}~h=2 ~\text{and}~ i=k-3.	
			\end{cases}\]
			
			\item[\upshape(10)]  If $\{x,y,z\}=\{i,i+1,j\}$, where  $3\leq i  \leq j-4 \leq k-4$, then $\left|hA\right|=h(k+2)+1$.
		\end{enumerate}
	\end{prop}
	
	\begin{proof}
		\begin{enumerate}
			\item [\upshape(1).] If $\{x,y,z\}=\{1,2,4\}$, then $A=\{0,3\} \cup [5,k+2]$. It is easy to see that \[ [5h,h(k+2)] \subseteq hA.\]
			Suppose $m$ be a positive integer such that $5\leq m \leq 5h-1$, and $r(m,5)$ be the least nonnegative residue of $m$ modulo $5$. Note that $5+r(m,5) \in 2A$. So, by Lemma \ref{Lemma IV}, we have  $m \in hA$. Hence, \[hA=\{0,3\} \cup [5,h(k+2)],\] and \[\left|hA\right|=h(k+2)-2 \text{ for } h\geq 2 \text{ and } k \geq 5.\]
			
			\item [\upshape(2).]If  $\{x,y,z\}=\{1,2, k+1\}$, then $A=\{0\} \cup [3,k] \cup \{k+2\}$.
			Since $k\geq 5$, by Proposition \ref{Proposition I}, we have \[h(\{0\} \cup [3,k]) =\{0\} \cup [3,hk] \subseteq hA.\]
			Also, $h([3,k] \cup \{k+2\}) \subseteq hA$, where
			\begin{align*}
				h( [3,k] \cup \{k+2\})&= h((k+2)- (\{0\} \cup [2,k-1]))\\
				&= h(k+2)-h(\lbrace 0 \rbrace \cup [2,k-1])\\
				&=h(k+2)-(\lbrace 0 \rbrace \cup [2,h(k-1)]) ~(\text{by Proposition \ref{Proposition I})}\\
				&= [3h,h(k+2)-2] \cup \{ h(k+2) \}.
			\end{align*}
			Hence, \[hA= \lbrace 0 \rbrace \cup [3,h(k+2)-2] \cup \{ h(k+2) \},\] and \[\left|hA\right|=h(k+2)-2, ~ \text{for all}~ h\geq 2 \text{  and } k \geq 5.\]

			\item[\upshape(3).]	 If  $\{x,y,z\}=\{1,2,i\}$, where $5 \leq i \leq k$, then $A=\{ 0 \} \cup [3,i-1] \cup [i+1,k+2] $. For $h=2$, it is easy to see that 
			\[2A=\begin{cases}
				\lbrace 0,3,4\rbrace \cup [6,2(k+2)] & \text{ if } i=5 \\
				\lbrace 0 \rbrace \cup [3,2(k+2)] & \text{ if } i\geq 6.
			\end{cases}\] Now, assume $h\geq 3$. Using  Lemma \ref{Lemma III}, we have 
			\begin{align*}
				h([3,i-1] \cup [i+1,k+2] )&=h(([0,i-4] \cup [i-2,k-1])  + 3)\\
				&=h([0,i-4] \cup [i-2,k-1] ) +3h\\
				&=[0,h(k-1)]+3h \\
				&=	[3h,h(k+2)].
			\end{align*}
			Next, we show $[6,3h-1] \subseteq hA$. Suppose $m$ be a positive integer such that  $6 \leq  m \leq 3h-1$.   Then \[2 \leq \Big\lfloor \frac{m}{3} \Big\rfloor \leq h-1 \ \text{and} \ r(m,3) \in [0,2].\]
			Note that,  $3+r(m,3) \in 2A$ for $r(m,3)\in \{0,1\}$, and so by Lemma \ref{Lemma IV}, $m \in hA$. For $r(m,3)=2$, we write  
			\[
			m=\underbrace{3+ \cdots +3}_{{\lfloor \frac{m}{3}\rfloor}-2  ~times}+ 4+4+ \underbrace{0+\cdots+0}_{h-{\lfloor \frac{m}{3}\rfloor} ~times} \in hA.
			\]  Hence,
			\[
			hA =
			\begin{cases}
				\{0,3,4\} \cup [6,h(k+2)],  & \text{ if } i=5;\\
				\{0\} \cup [3,h(k+2)], & \text{ if } i \neq 5,
			\end{cases}
			\]
			and
			\[
			|hA| =
			\begin{cases}
				h(k+2)-2, & \text{ if } i=5;\\
				h(k+2)-1, & \text{ if } i \neq 5.
			\end{cases}
			\]
			
			\item [\upshape(4).]	If $\{x,y,z\}=\{2,3,5\}$, then $A= \{0,1,4\} \cup [6,k+2]$. We have
			\begin{align*}
				h(\{4\} \cup [6,k+2])&=	h((\lbrace 0 \rbrace \cup [2,k-2]) +4)\\
				&=h(\lbrace 0 \rbrace \cup [2,k-2]) +4h\\
				&=(\lbrace 0 \rbrace \cup [2,h(k-2)])+4h \text{ (by Proposition \ref{Proposition I}) }\\
				&=\lbrace 4h \rbrace \cup [4h+2,h(k+2)].
			\end{align*}
			
			For $h \geq 3$, 	\[
			4h+1=\underbrace{4+ \cdots +4}_{(h-3)  ~times}+ 6+6+1 \in hA.
			\]
			Suppose $m$ be a positive integer such that $4\leq m \leq 4h-1$.
			Clearly $4+r(m,4) \in 2A$ for each $m$. So, by Lemma \ref{Lemma IV}, $m \in hA$.
			Therefore,		
			\[
			hA =
			\begin{cases}
				[0,h(k+2)],  & \text{if $h \geq 3$ and $k \geq 5$};\\
				[0,2] \cup [4, 2(k+2)], & \text{if $h=2$ and $k\geq 6$};\\
				[0,2] \cup [4,8] \cup [10,14], & \text{if $h=2$ and $ k=5$},
			\end{cases}
			\] and hence,	\[	\left|hA\right| =
			\begin{cases}
				h(k+2)+1, & \text{if $h \geq 3$ and $k \geq 5$};\\
				2k+4, & \text{if $h=2$ and $k\geq 6$};\\
				13, & \text{if $h=2$ and $ k=5$}.
			\end{cases}
			\]

			\item [\upshape(5).]	If $\{x,y,z\} = \{2,3,k+1\}$, then $A=\{0,1\} \cup [4,k] \cup \lbrace k+2 \rbrace$. For $h\geq 3$, $h([0,1] \cup [4,k])=[0,hk]$ (by Lemma \ref{Lemma III})  and
			\begin{align*}
				h([4,k] \cup \lbrace k+2 \rbrace)&=h((k+2)-(\lbrace 0 \rbrace \cup [2, k-2]))\\
				&=h(k+2)-h(\lbrace 0 \rbrace \cup [2, k-2])\\
				&=h(k+2)-(\lbrace 0 \rbrace \cup [2, h(k-2)]) \text{ (by Proposition \ref{Proposition I})}\\
				&=[4h,h(k+2)-2] \cup \{ h(k+2) \}.
			\end{align*}
			Therefore,\[
			hA =
			\begin{cases}
				[0,h(k+2)-2] \cup \{ h(k+2) \},  & \text{if $h \geq 3$ and $k \geq 5$};\\
				[0,2] \cup [4, 2k+2] \cup \{ 2k+4 \}, & \text{if $h=2$ and $k\geq 5$},
			\end{cases}
			\] and hence,
			\[
			\left|hA\right| =
			\begin{cases}
				h(k+2), & \text{if $h \geq 3$ and $k \geq 5$};\\
				2k+3, & \text{if $h=2$ and $ k\geq 5$}.
			\end{cases}
			\]

			\item [\upshape(6).]	If $\{ x,y,z \}=\{ 2,3,i \}$,  where $6\leq i \leq k$, then $A=\{0,1\} \cup [4, i-1] \cup [i+1,k+2]$.
			By Lemma \ref{Lemma III}, we have
			\[
			hA=
			\begin{cases}
				[0,h(k+2)],  & \text{if $h \geq 3$ and $k \geq 6$};\\
				[0,2] \cup [4, 2k+4],  & \text{if $h=2$ and $k \geq 6$}, 
			\end{cases}
			\] and
			\[
			\left|hA\right| =
			\begin{cases}
				h(k+2)+1, & \text{if $h \geq 3$ and $k \geq 6$};\\
				2k+4 & \text{if $h=2$ and $ k\geq 6$}.
			\end{cases}
			\]

			\item [\upshape(7).]	If  $\{x,y,z\}=\{k-2,k-1,k+1\}$, then $A =[0,k-3] \cup \{k,k+2\}$. For $h=2$,  we have 	\[2A=[0,2k] \cup \{ 2k+2,2k+4 \}.\]
			Assume $h\geq 3$. Suppose $m$ be a positive integer such that  $k+2 \leq m \leq h(k+2)-1$. Then 
			\[1\leq \Big\lfloor \frac{m}{k+2}\Big\rfloor \leq h-1 \ \text{and} \ r(m,k+2) \in [0,k+1].\] 
			For  $r(m,k+2) \in [ 0, k-2] \cup  \{ k \}$, we have $k+2+r(m,k+2)\in 2A$, and therefore by Lemma \ref{Lemma IV}, $m \in hA$.  For $r(m,k+2) \in \{k-1,k+1\}$ and $1\leq \Big\lfloor \dfrac{m}{k+2}\Big\rfloor\leq h-2$, we writing $m$ respectively by 		
			\[ 	m =
			\underbrace{(k+2)+ \cdots +(k+2)}_{{\lfloor \frac{m}{k+2}\rfloor} \ times} +(k-3)+2  + \underbrace{0+\cdots+0}_{h-2-{\lfloor \frac{m}{k+2}\rfloor} ~times} \in hA,
			\]
			and
			\[  	m =
			\underbrace{(k+2)+ \cdots +(k+2)}_{\lfloor \frac{m}{k+2}\rfloor \ times} +k+1  + \underbrace{0+\cdots+0}_{h-2-\lfloor \frac{m}{k+2}\rfloor ~times} \in hA. \]  If $\lfloor \frac{m}{k+2} \rfloor =h-1$ and $r(m,k+2) \in \{k-1,k+1\}$, then $m \in \{h(k+2)-3, h(k+2)-1\}$. It is easy to check that $\{h(k+2)-3, h(k+2)-1\} \subseteq hA$ for $h \geq 3$. 
			Hence, $$hA=[0,h(k+2)-4] \cup \{h(k+2)-2, h(k+2)\},$$ and 
			\[|hA|=h(k+2)-1 \text{  for } h\geq 2 \text{ and } k \geq 5.\]

			\item [\upshape(8).] If   $\{x,y,z\}=\{i,i+1,k+1 \}$, where $3\leq i \leq k-3$, then $A =[0,i-1] \cup[i+2,k] \cup \{k+2\}$. For $h=2$ and $k\geq 6$, it is easy to see that $2A=[0,2k+2] \cup \{2k+4\}$. Now, assume $h\geq 3$. Using Lemma \ref{Lemma III}, we obtain \[h([0,i-1] \cup [i+2,k])=[0,hk] ~\text{for}~  h \geq 3.\] Also, $h([i+2,k] \cup \{ k+2 \} )=[h(i+2),h(k+2)-2] \cup \{ h(k+2) \}$ by Proposition \ref{Proposition I}. Hence, $$hA=[0,h(k+2)-2] \cup \{h(k+2)\},$$  and
			\[|hA|=h(k+2) \text{ for } h\geq 2 ~\text{and}~ k\geq 6.\]

			\item [\upshape(9).]	If $\{x,y,z\}=\{i,i+1,i+3\}$, where $3 \leq i \leq k-3$, then $A=[0,i-1] \cup \{i+2\} \cup [i+4,k+2]$. For $h=2$, it is easy to see that \[2A = \begin{cases}
				[0,2k+4], & ~\text{if}~ 3\leq i \leq k-4;\\
				[0,2k-2] \cup [2k,2k+4], &~\text{if}~i=k-3.
			\end{cases}\] Now, assume $h\geq 3$.  We have
			\begin{align*}
				h(\{ i+2 \} \cup [i+4,k+2])&=	h((\{ 0 \} \cup [2,k-i]) +i+2)\\
				&=h(\{ 0 \} \cup [2,k-i]) +h(i+2)\\
				&=\{ h(i+2) \} \cup [h(i+2)+2,h(k+2)]  (\text{by Proposition \ref{Proposition I}})
			\end{align*}
			and
			\[h(i+2)+1=\underbrace{(i+2)+\cdots+ (i+2)}_{(h-3) ~ times}+(i+4)+(i+4)+(i-1) \in hA.\]
			Also,
			\begin{align*}
				h([0,i-1] \cup \{ i+2 \})&= h(\{i+2\}-(\{ 0 \} \cup [3,i+2]))\\
				&=h(i+2)-h(\{ 0 \} \cup [3,i+2])\\
				&=h(i+2)-(\{ 0 \} \cup [3,h(i+2)])\\
				&=[0,h(i+2)-3] \cup \{h(i+2)\},
			\end{align*}
			\[h(i+2)-2=\underbrace{(i+2)+\cdots+ (i+2)}_{(h-3)~ times} + (i-2)+(i+2)+(i+4) \in hA,\] and
			\[h(i+2)-1=\underbrace{(i+3)+\cdots+ (i+3)}_{(h-3)~times} + (i-1)+(i+2)+(i+4) \in hA, \text{ for } h \geq 3.\] Thus, $hA =[0,h(k+2)]$ for $h\geq 3$. Hence, \[ hA = \begin{cases}
				[0,h(k+2)], & ~\text{if}~ h\geq 3;\\
				[0,2k+4], & ~\text{if}~h=2 ~\text{and}~ 3\leq i \leq k-4;\\
				[0,2k-2] \cup [2k,2k+4], &~\text{if}~h=2 ~\text{and}~ i=k-3,
			\end{cases}\]
			and
			\[ |hA| = \begin{cases}
				h(k+2)+1, & ~\text{if}~ h\geq 3;\\
				2k+5, & ~\text{if}~ h=2 ~\text{and}~ 3\leq i \leq k-4;\\
				2k+4, &~\text{if}~h=2 ~\text{and}~ i=k-3.	
			\end{cases}\]

			\item [\upshape(10).]If $\{x,y,z\}=\{i,i+1,j\}$, where $3\leq i\leq j-4 \leq k-4$, then $A=[0,i-1] \cup [i+2,j-1] \cup [j+1,k+2]$. So, $hA=[0,h(k+2)]$ for $h \geq 2$ (by Lemma \ref{Lemma III}). Hence, $\left|hA\right|=h(k+2)+1$.
		\end{enumerate}
		This completes the proof of the proposition.
	\end{proof}

	\begin{prop}\label{Proposition IV}
		Let $h\geq 2$ and $k \geq 5$  be positive integers. Let \[A =[0,k+2] \setminus \{x,y,z\},\] where $1 \leq x\leq y -2 \leq z-4 \leq k-3$.
		\begin{enumerate}
			\item[\upshape(1)] If $\{x,y,z\}$ is one of the sets $\{1,3,5\}$  and $\{k-3,k-1,k+1\}$, then $\left|hA\right|= h(k+2)-2$.
			
			\item[\upshape(2)] If $\{x,y,z\}$ is one of the sets $\{1,k-1,k+1\}$ and $\{1,3,k+1\}$, then $\left|hA\right|= h(k+2)-2$.
			
			\item[\upshape(3)] If $\{x,y,z\}$ is one of the sets $\{1,3,i\}$ where $6 \leq i \leq k$, and  $\{i,k-1,k+1\}$ where $2 \leq i \leq k-4 $,  then $\left|hA\right|= h(k+2)-1$.
			
			\item[\upshape(4)]  If $\{x,y,z\} = \{i,i+2,i+4\}$, where $2\leq i \leq k-4$, then  $\left|hA\right|= h(k+2)+1$.
			
			\item[\upshape(5)] If $\{x,y,z\}$ is one of the sets $\{i,i+2,k+1\}$  where $2 \leq i \leq k-4$,  and $\{1,i,i+2\}$ where $4 \leq i \leq k-2$, then $\left|hA\right|= h(k+2)$.
			
			\item[\upshape(6)] If $\{x,y,z\}$ is one of the sets $\{i,i+2,j\}$  where $ 2 \leq i \leq j-5 \leq k-5$, and  $\{i,j,j+2\}$ where $2 \leq i \leq j-3 \leq k-5$,  then $\left|hA\right|= h(k+2) +1$.
			
			\item[\upshape(7)] If $\{x,y,z\}= \{1,i,k+1\}$, where $4 \leq i \leq k-2$,  then  $|hA|=h(k+2)-1$.
			
			\item[\upshape(8)] If $\{x,y,z\}$ is one of the sets  $\{1,i,j\}$ where $4 \leq i \leq j-3 \leq k-3$, and $\{i,j,k+1\}$ where $2 \leq i \leq j-3 \leq k-5$, then  $|hA|=h(k+2)$.
			
			\item[\upshape(9)] If $2\leq x \leq y-3 \leq z-6 \leq k-6$,  then  $|hA|=h(k+2)+1$.
		\end{enumerate}
	\end{prop}
	
	\begin{proof}
		\begin{enumerate}
			\item [\upshape(1).]
			\begin{enumerate}
				\item [\upshape(a)]  If $\{x,y,z\} = \{1,3,5\}$, then $A=\{0,2,4\} \cup [6,k+2]$. Using Proposition \ref{Proposition I}, we have
				\begin{align*}
					h(\lbrace 4 \rbrace \cup [6,k+2]) &= h(\lbrace 0 \rbrace \cup [2,k-2]+4)\\
					& = h(\lbrace 0 \rbrace \cup [2,k-2])+4h\\
					& = (\{0\} \cup [2,h(k-2)])+4h\\
					&= \lbrace 4h \rbrace \cup [4h+2, h(k+2)].
				\end{align*}
				Also,\[4h+1 = \underbrace{4+ \cdots +4}_{(h-2) \ times}+7+2 \in hA ~ \text{for all $ h \geq 2$}.\]
				Next, we show that $[4,4h-1]\subseteq hA$. Suppose $m$ be a positive integer such that $4\leq m \leq 4h-1$. Note that $ 4+r(m,4) \in 2A$ for $r(m,4) \in \{0,2,3\}$, and therefore by Lemma \ref{Lemma IV}, $m\in hA$. For $r(m,4)=1$ and $\lfloor \frac{m}{4}\rfloor \geq 2$, we writing $m$ in the following way:
				\[m = \underbrace{4+ \cdots +4}_{\lfloor \frac{m}{4}\rfloor-2 \ times}+7+2 \in hA.\]
				Note also that, $r(m,4) = 1$ and $\lfloor \frac{m}{4}\rfloor =1$ gives $m=5 \notin hA$. Hence,
				\[ hA= \lbrace 0,2,4 \rbrace \cup [6,h(k+2)],\] and \[|hA|= h(k+2)-2.\]
				
				\item [\upshape(b)] If $\{x,y,z\} = \{k-3,k-1,k+1\}$, then we have $$A = [0,k-4] \cup \lbrace k-2,k,k+2 \rbrace=(k+2) - (\lbrace 0,2,4 \rbrace \cup [6,k+2]).$$ Since the cardinality of $hA$ is translation invariant, so  $\left|hA\right|= h(k+2)-2$.
			\end{enumerate}
			
			\item[\upshape(2).] \begin{enumerate}
				\item [\upshape(a)]  If $\{x,y,z\}=\{1,k-1,k+1\}$, then $A= \{0\} \cup [2,k-2] \cup \{k,k+2\}$. For $h=2$, it is easy to see that \[2A=\{0\} \cup[2,2k] \cup \{2k+2,2k+4\}.\] Now, assume $h\geq 3$. Next, we show that $[k+2,(h-1)(k+2)-1] \subseteq hA$. Suppose $m$ be a positive integer such that $k+2 \leq m \leq (h-1)(k+2)-1$. Then \[1 \leq  \Big\lfloor \frac{m}{k+2} \Big\rfloor\leq h-2 \ \text{and} \ r(m,k+2) \in [0,k+1].\]
				Note that, $(k+2) +r(m,k+2) \in 2A$ for $r(m,k+2) \in [0, k-2] \cup \{k\}$, so by Lemma \ref{Lemma IV}, $m \in hA$.  For $r(m,k+2) \in \{k-1,k+1\}$, we writing  $m$ respectively, by
				\[ m= \underbrace{(k+2)+ \cdots +(k+2)}_{\lfloor \frac{m}{k+2}\rfloor-1~times} +k+(k-2)+3 +\underbrace{ 0+ \cdots +0}_{h-{\lfloor \frac{m}{k+2}\rfloor}-2~times}\in hA,\]
				\[ m= \underbrace{(k+2)+ \cdots +(k+2)}_{\lfloor \frac{m}{k+2}\rfloor-1~times} +(k+2)+(k-2)+3 +\underbrace{ 0+ \cdots +0}_{h-\lfloor \frac{m}{k+2}\rfloor-2~times} \in hA.\]
				It is easy to check that $(h-1)(k+2), (h-1)(k+2)+1, \ldots, h(k+2)-4, h(k+2)-2,h(k+2)$, all are in $hA$.  If $i\geq 2$, then $(h-i)(k+2)+ik\leq h(k+2)-4$. So, $h(k+2)-3$ and $h(k+2)-1$ are not in $hA$.  Hence,
				\[hA=\{0\} \cup [2,h(k+2)-4] \cup \{h(k+2)-2,h(k+2)\},\] and  \[\left|hA\right|= h(k+2)-2.\]
				
				\item [\upshape(b)] If $\{x,y,z\} = \{1,3,k+1\}$, then we have $$A= \{0,2\} \cup [4,k] \cup \{k+2\} =  (k+2)-(\{0\} \cup [2,k-2] \cup \{k,k+2\}).$$ Since cardinality of $hA$ is translation invariant, so $\left|hA\right|= h(k+2)-2.$
			\end{enumerate}
			
			\item[\upshape(3).] \begin{enumerate}
				\item [\upshape(a)]  If $\{x,y,z\}=\{1,3,i\}$, where $6 \leq i \leq k$, then $A=\{0,2\} \cup [4,i-1] \cup [i+1,k+2]$. We have,
				\begin{align*}	
					h([4,i-1] \cup [i+1,k+2]) &= h([0,i-5] \cup [i-3,k-2]) +4h\\
					&=[4h,h(k+2)],  \ (\text{by Lemma \ref{Lemma III}}).	
				\end{align*}
				Suppose $m$ be positive integer such that $4 \leq m \leq 4h-1$. Then $4 + r(m,4) \in 2A$, so by Lemma \ref{Lemma IV}, $m\in hA$. Hence, \[hA=\{0,2\} \cup [4,h(k+2)],\] and \[\left|hA\right|=h(k+2)-1.\] 
				
				\item [\upshape(b)] If $\{x,y,z\}=\{j,k-1,k+1\}$, where $2 \leq j \leq k-4$, then we have $$A=(k+2)- (\{0,2\} \cup [4,i-1] \cup [i+1,k+2]),$$ where $6 \leq i=k+2-j \leq k$. Since cardinality of $hA$ is translation invariant, so  $\left|hA\right|=h(k+2)-1.$
			\end{enumerate}

			\item[\upshape(4).]  If $\{x,y,z\} = \{i,i+2,i+4\}$, where $2\leq i \leq k-4$, then $A= [0,i-1] \cup \{ i+1, i+3 \} \cup [i+5,k+2]$. Using  Proposition \ref{Proposition I}, we have 
			\begin{align*}
				h(\{i+3\} \cup [i+5,k+2])&=h(\{0\} \cup [2,k-i-1]+(i+3))\\
				&=h(\{0\} \cup [2,k-i-1])+h(i+3)\\
				&=(\{0\} \cup [2,h(k-i-1)])+h(i+3) \\
				&=\{h(i+3)\} \cup [h(i+3)+2,h(k+2)].
			\end{align*}
			Moreover, $h(i+3)+1=(h-2)(i+3)+(i+6)+(i+1) \in hA$. Next, we show that $[i+3, h(i+3)-1] \subseteq hA$. Suppose $m$ be a positive integer such that $i+3 \leq m \leq h(i+3)-1$. It is easy to see that $(i+3) + r(m,i+3) \in 2A$ for $r(m,i+3) \in [0, i-1] \cup  \{i+1\}$. 
			If $r(m,i+3) \in \{i, i+2\}$, then we have \[(i+3)+i=(i+5)+(i-2) \in 2A\] and  \[(i+3)+(i+2)=(i+6)+(i-1) \in 2A.\]
			Thus, $(i+3) + r(m,i+3) \in 2A$, and  by Lemma \ref{Lemma IV}, $m\in hA$. Hence, \[hA=[0,h(k+2)],\] and \[\left|hA\right|=h(k+2)+1    ~ \text{for}~ h \geq 2 ~ \text{and}~ k\geq 6.\]

			\item[\upshape(5).] \begin{enumerate}
				\item [\upshape(a)]  If $\{x,y,z\}= \{i,i+2,k+1\}$,  where $2 \leq i \leq k-4$, then $A=[0,i-1] \cup \lbrace i+1 \rbrace \cup [i+3,k] \cup \lbrace k+2 \rbrace$. We have
				\begin{align*}
					h([i+3,k] \cup \lbrace k+2 \rbrace) &=h((k+2)- (\lbrace 0 \rbrace \cup [2,k-i-1]))\\
					&= h(k+2)-(\lbrace 0 \rbrace \cup [2,h(k-i-1)])\\
					&=[h(i+3),h(k+2)-2] \cup \lbrace h(k+2) \rbrace \subseteq hA,
				\end{align*}
				and $[0,i+2] \subseteq hA$. Next, we show that $[i+3,h(i+3)-1] \subseteq hA$. Suppose $m$ be a positive integer such that  $i+3 \leq m \leq h(i+3)-1$. It is easy to see that $(i+3)+r(m,i+3) \in 2A$ for $r(m,i+3) \in \{0, 1, \ldots, i-1, i+1\}$. If $r(m,i+3) \in \{i,i+2\}$, then we have \[ (i+3)+i=(i+4)+(i-1) \in 2A\] and \[(i+3)+(i+2)=(i+4)+(i+1) \in 2A.\] So, by Lemma \ref{Lemma IV}, $m\in hA$. Hence \[hA=[0,h(k+2)-2] \cup \lbrace h(k+2) \rbrace\] and
				\[\left|hA\right|=h(k+2).\]
				
				\item [\upshape(b)] If $\{x,y,z\}= \{1,j,j+2\}$, where $4 \leq j \leq k-2$, then we have $$A = \{k+2\}-([0,k+2] \setminus \{i, i+2,k+1\}),$$ where $2 \leq i=k-j \leq k-4 $. Since sumset $hA$ is translation invariant, so  $\left|hA\right|=h(k+2).$
			\end{enumerate}
			
			\item [\upshape(6).]\begin{enumerate}
				\item [\upshape(a)]  If $\{x,y,z\}= \{i,i+2,j\}$,  where $ 2 \leq i \leq j-5 \leq k-5$. Then  $A=[0,i-1] \cup \lbrace i+1 \rbrace \cup [i+3,j-1] \cup [j+1,k+2]$. Since $i+3 \leq j-2$, we have  $h( [i+3,j-1] \cup [j+1,k+2])=[h(i+3),h(k+2)]$, by Lemma \ref{Lemma III}. The arguments similar to the ones in case 5, show that $m\in hA$, where  $m \in [i+3, h(i+3)-1]$. Hence, $hA=[0,h(k+2)]$ and $|hA|=h(k+2)+1$.
				
				\item [\upshape(b)] If $\{x,y,z\} = \{i,j,j+2\}$, where $2 \leq i \leq j-3 \leq k-5$, then $A$ is translation of the set $[0,i_{0}-1] \cup \lbrace i_{0}+1 \rbrace \cup [i_{0}+3,j-1] \cup [j_{0}+1,k+2]$ where $ 2 \leq i_{0} \leq j_{0}-5 \leq k-5$. Therefore, $|hA|=h(k+2)+1$.
				
			\end{enumerate}
			\item [\upshape(7).] If $\{x,y,z\}= \{1,i,k+1\}$, where $4 \leq i \leq k-2$, then $A=\{0\} \cup [2,i-1] \cup [i+1,k] \cup \{k+2\}$. We have,
			\[h(\{0\} \cup [2,i-1])=\{0\} \cup [2,h(i-1)] \ (\text{by Proposition } \ref{Proposition I}),\] \[h([2,i-1] \cup [i+1,k])=[2h,hk] \ (\text{by Lemma } \ref{Lemma III}),\] and
			\begin{align*}
				h([i+1,k] \cup \{k+2\})& =h(k+2)-h(\{0\} \cup [2,k+1-i])\\
				& =[h(i+1),h(k+2)-2] \cup \{h{k+2}\}  \ (\text{by Proposition} \  \ref{Proposition I}).	
			\end{align*}
			Hence, $hA= \lbrace 0 \rbrace \cup [2,h(k+2)-2] \cup \lbrace h(k+2) \rbrace$ and $|hA|=h(k+2)-1$.
			
			\item [\upshape(8).] \begin{enumerate}
				\item [\upshape(a)]  If $\{x,y,z\} = \{1,i,j\}$, where $4 \leq i\leq j-3 \leq k-3$, then $A=\{0\} \cup [2,i-1] \cup [i+1,j-1] \cup [j+1,k+2]$. The arguments in this case are similar to the one in case 7. Hence,  $hA= \lbrace 0 \rbrace \cup [2,h(k+2)]$ and $|hA|=h(k+2)$.
				
				\item [\upshape(b)]   If $\{x,y,z\}= \{i,j,k+1\}$, where $2 \leq i \leq j-3 \leq k-5$, then $A$ is a translation of $\{0\} \cup [2,i_{0}-1] \cup [i_{0}+1,j_{0}-1] \cup [j_{0}+1,k+2]$,  where $4 \leq i_{0}\leq j_{0}-3 \leq k-3$. Hence, $|hA|=h(k+2)$.
			\end{enumerate}
			\item [\upshape(9).] If  $2\leq x \leq y-3 \leq z-6 \leq k-6$, then $A=[0,x-1] \cup [x+1,y-1] \cup [y+1,z-1] \cup [z+1,k+2]$.  Lemma \ref{Lemma III} implies that $hA=  [0,h(k+2)]$ and hence, $|hA|=h(k+2)+1$.
		\end{enumerate}
	\end{proof}

	\section{Proof of Theorem \ref{Main Theorem}}\label{Mainsection}
	
	\begin{proof}
		Let $k\geq 6$ and $ \left|2A\right|=2k+2 $. Then $ \left|2A\right|=2k+2 \leq 3k-4$ and   by Theorem \ref{Theorem Freiman II},  $A\subseteq [0,k+2]$. In  rest of the cases, $A\subseteq [0,k+2]$ due to Lemma \ref{Lemma I}. Now, Theorem  \ref{Theorem Tang II}, Proposition \ref{Proposition II}, \ref{Proposition III} and \ref{Proposition IV} give the structure of $A$. This completes the proof of Theorem \ref{Main Theorem}.
	\end{proof}

	\section{Conclusion}\label{conclusion}
	We know that,  $\left|3A\right|=3k-2$ if and only if  $ \left|2A\right| =  2k-1$, the reason being $A$ is an arithmetic progression. Theorem \ref{Theorem Lev} gives a relation between the sizes of $hA$ and $(h-1)A$. Let $k\geq 6$. Then by Theorem \ref{Theorem Tang I}, \ref{Theorem Tang II} and \ref{Main Theorem}, we have the following:
	\begin{enumerate}
		
		\item  $3k+1 \leq \left|3A\right| \leq 3k+2$ if and only if $ \left|2A\right|  = 2k+1$.
		
		\item If $\left|3A\right|=3k+3$, then $ \left|2A\right|=2k+2$.
		
		\item If $\left|3A\right|=3k+4$, then $2k+2 \leq \left|2A\right|  \leq 2k+3$.
	\end{enumerate}
	
	A remark of Tang and Xing (\cite{TANG2021}, Remark 1.3), states that there is no set $A$ such that $|3A|=3|A|-1$. Similar to this remark, we have the following observation.
	
	\begin{enumerate}
		
		\item If $h\geq 3$ and $ k\geq 5$, then there is no set $A$ such that $hk-h+2 \leq \left|hA\right|  \leq hk-1$ (See Lemma 2.1 and Proposition 3.1 in \cite{TANG2021}).
		
		\item If $h\geq 4$ and $ k\geq 5$, then there is no set $A$ such that $hk+2 \leq \left|hA\right|  \leq hk+h-2$ (See Lemma 2.1 and Proposition 3.1 in \cite{TANG2021}).

		\item If $h\geq 5$ and $ k\geq 5$, then there is no set $A$ such that $hk+h+2 \leq \left|hA\right|  \leq hk+2h-3$ (See Lemma 2.1 and Proposition 3.2 - 3.4 in \cite{TANG2021}).
		
		\item If $h\geq i+4$ where $i \in [2,h-4]$ and $ k\geq 5$, then there is no set $A$ such that $hk+2h+2 \leq \left|hA\right| =hk+2h+i \leq hk+3h-4$ (See Lemma \ref{Lemma I} and Proposition \ref{Proposition II} - \ref{Proposition IV} ).
	\end{enumerate}

	\section*{Acknowledgment}
	
	The first author would like to thank to the Council of Scientific and Industrial Research (CSIR), India for providing the grant to carry out the research with Grant No. 09/143(0925)/2018-EMR-I and the second author wishes to thank to the Science and
	Engineering Research Board (SERB), India for providing the grant with Grant No.
	MTR/2019/000299.

	\bibliographystyle{amsplain}

\begin{thebibliography}{5}
		
		\bibitem{GAF1959} G. A. Frieman, {\it On the addition of finite sets I}, Izv. Vysh. Uchebn. Zaved. Mat. \textbf{13}(6) (1959), 202--213.
		
		\bibitem{GAF1973} G. A. Freiman, {\it Foundations of structural theory of set addition}, vol. {\bf 37}, Translations of Mathematical Monographs, American Mathematical Society, Provedience, R.I., 1973.
		
		
		
		\bibitem{VFL1996} V. F. Lev, \textit{Structure theorem for multiple addition and the Frobenius problem}, J. Number Theory \textbf{58}(1) (1996), 79--88.
		
		
		\bibitem{MBN1996} M.  B. Nathanson, {\it Additive Number Theory: inverse problems and the geometry of sumsets}, Springer, 1996.
		
		\bibitem{TANG2021} M. Tang and Y. Xing, \textit{Some inverse results of sumsets}, Bull. Korean Math. Soc. \textbf{58} No. 2 (2021), 305--313.
		
	\end{thebibliography}

\end{document}